\documentclass{amsart}

\usepackage[english]{babel} 
\usepackage[utf8]{inputenc}  %% les accents dans le fichier.tex
\usepackage[T1]{fontenc}       %% Pour la césure des mots accentués

\usepackage{amssymb,amsfonts,amsthm,amsmath}

\usepackage{enumerate}
\usepackage{xspace}
\usepackage{comment}
\usepackage{graphicx}
\usepackage{multirow}
\usepackage[all]{xy}
\usepackage[hyphens]{url}
\usepackage{tikz}
\usepackage{lipsum}
\usetikzlibrary{arrows}

\newcommand{\F}{{\mathbb F}}
\newcommand{\Q}{{\mathbb Q}}
\newcommand{\Z}{{\mathbb Z}}
\newcommand{\PP}{{\mathbb P}}
\newcommand{\ZTF}{{\Z/2\Z\times\Z/4\Z}}

\newcommand{\ds}{\displaystyle}

\newcommand\funding[1]{%
  \begingroup
  \renewcommand\thefootnote{}\footnote{#1}%
  \addtocounter{footnote}{-1}%
  \endgroup
}

\DeclareMathOperator{\Gal}{Gal}
\DeclareMathOperator{\GL}{GL}
\DeclareMathOperator{\Frob}{Frobenius}
\DeclareMathOperator{\Spec}{Spec}
\DeclareMathOperator{\Prob}{\mathbb{P}}
\DeclareMathOperator{\fix}{Fix}
\DeclareMathOperator{\Aut}{Aut}

\DeclareMathOperator{\Dec}{Dec}

\DeclareMathOperator{\Res}{Res}
\DeclareMathOperator{\new}{new}
\DeclareMathOperator{\lift}{Lift}
\DeclareMathOperator{\Id}{Id}

\theoremstyle{plain}
\newtheorem{theo}{Theorem}[section]

\newtheorem{prop}[theo]{Proposition}
\newtheorem{cor}[theo]{Corollary}

\theoremstyle{definition}
\newtheorem{deff}[theo]{Definition}
\newtheorem{ex}[theo]{Example}

\theoremstyle{remark}
\newtheorem{rem}[theo]{Remark}
\newtheorem{nott}[theo]{Notation}

\title[Finding ECM-friendly curves through a study of Galois properties]
{Finding ECM-friendly curves\\through a study of Galois properties}

\author[Barbulescu]{Razvan Barbulescu}
\address{Université de Lorraine, CNRS, INRIA, France}
\curraddr{}
\email{}
\thanks{}

\author[Bos]{Joppe W. Bos}
\address{Microsoft Research, One Microsoft Way, Redmond, WA 98052, USA}
\curraddr{}
\email{}
\thanks{}

\author[Bouvier]{Cyril Bouvier}
\address{ENS Paris, Université de Lorraine, CNRS, INRIA, France}
\curraddr{}
\email{}
\thanks{}

\author[Kleinjung]{Thorsten Kleinjung}
\address{Laboratory for Cryptologic Algorithms, EPFL, Lausanne, Switzerland}
\curraddr{}
\email{}
\thanks{}

\author[Montgomery]{Peter L. Montgomery}
\address{Microsoft Research, One Microsoft Way, Redmond, WA 98052, USA}
\curraddr{}
\email{}
\thanks{}


\begin{document}
\begin{abstract}
In this paper we prove some divisibility properties of the cardinality of 
elliptic curve groups modulo primes. These proofs explain the good behavior of 
certain parameters when using Montgomery or Edwards curves in the setting
of the elliptic curve method (ECM) for integer factorization. 
The ideas behind the proofs help us to find new infinite families of elliptic
curves with good division properties increasing the success
probability of ECM.
\end{abstract}

\maketitle

\section{Introduction}
The elliptic curve method (ECM) for integer factorization~\cite{ECM} is the 
asymptotically fastest known method for finding relatively small factors $p$ of large
integers $N$. In practice, ECM is used, on the one hand, to factor large integers.
For instance, the 2011 ECM-record is a 241-bit factor of $2^{1181}-1$~\cite{MersenneARITH}.
On the other hand, ECM is used to factor many small (100 to 200 bits) 
integers as part of the number field sieve~\cite{PolSNFS,NFS,cado}, 
the most efficient general purpose integer factorization method. 

Traditionally, the elliptic curve arithmetic used in ECM is implemented using 
Montgomery curves~\cite{Montgomery:1987} (e.g., in the widely-used 
GMP-ECM software~\cite{GMPECM}). Generalizing the work of Euler and Gauss, Edwards
introduced a new normal form for elliptic curves~\cite{edwards} which results
in a fast realization of the elliptic curve group operation in practice.
These Edwards curves have been generalized by Bernstein and Lange~\cite{Bernstein2007Asia}
for usage in cryptography. Bernstein et al. explored 
the possibility to use these curves in the ECM setting~\cite{cryptoeprint:2008:016}. 
After Hisil et al.~\cite{Hisil2008Asiacrypt} published a coordinate system which results in the
fastest known realization of curve arithmetic, a follow-up paper by Bernstein et al.
discusses the usage of the so-called ``$a=-1$'' twisted Edwards curves~\cite{starfish} in ECM.

It is common to construct or search for curves which have favorable properties.
The success of ECM depends on the smoothness of
the cardinality of the curve considered modulo the unknown prime divisor $p$ of
$N$. This usually means constructing curves with large torsion group over $\Q$ or finding 
curves such that the order of the elliptic curve, when considered modulo a 
family of primes, is always divisible by an additional factor. 
Examples are the Suyama construction~\cite{Suyama}, the curves proposed by 
Atkin and Morain~\cite{1993-atkin}, a translation of these techniques to 
Edwards curves~\cite{cryptoeprint:2008:016,starfish}, and a family of curves suitable 
for Cunningham numbers~\cite{BrierC10}.

\funding{This work was supported by the Swiss National Science Foundation under grant 
number 200020-132160 and by a PHC Germaine de Sta\"el grant.}

In this paper we study and prove divisibility properties of 
the cardinality of elliptic curves over prime fields. We do this by
studying properties of Galois groups of torsion points
using Chebotarev's theorem~\cite{Ne86}. Furthermore, we investigate 
some elliptic curve parameters for which ECM finds exceptionally many primes 
in practice, but which do not fit in any of the known cases of good torsion
properties. We prove this behavior and provide parametrizations for 
infinite families of elliptic curves with these properties. 

\section{Galois Properties of Torsion Points of Elliptic Curves}
In this section we give a systematic way to compute the
probability that the order of a given elliptic curve reduced by an arbitrary
prime is divisible by a certain prime power.   
\subsection{Torsion Properties of Elliptic Curves.}
\begin{deff}\label{Frob} Let $K$ be a finite Galois extension of $\Q$,
$p$ a prime and $\mathfrak{p}$ a prime ideal above $p$ with residue
field $k_{\mathfrak{p}}$.
The decomposition group $\Dec(\mathfrak{p})$ of $\mathfrak{p}$ is the
subgroup of $\Gal(K/\Q)$ which stabilizes $\mathfrak{p}$.
Call $\alpha^{(\mathfrak{p})}$ the canonical morphism from
$\Dec(\mathfrak{p})$ to $\Gal(k_{\mathfrak{p}}/\mathbb{F}_p)$ and
let $\phi_{\mathfrak{p}}$ be the Frobenius automorphism on the field $k_{\mathfrak{p}}$.
We define
$\Frob(p)=\bigcup_{\mathfrak{p}\mid p}
  (\alpha^{(\mathfrak{p})})^{-1}(\phi_{\mathfrak{p}})$.
\end{deff}
In order to state Chebotarev's theorem we say that a set $S$ of primes
admits a natural density equal to $\delta$ and we 
write $\Prob(S)=\delta$ if $\lim_{N\rightarrow \infty}\frac{\#(S\bigcap
\Pi(N))}{\#\Pi(N)}$ exists and equals $\delta$, where $\Pi(N)$ is the set of
primes up to $N$. If event($p$) is a property which can be defined for all
primes except a finite set (thus of null density), when we note $\Prob(\text{event}(p))$ we tacitly exclude the primes where
 event$(p)$ cannot be defined.  

\begin{theo}[Chebotarev,~\cite{Ne86}] \label{Chebotarev}
Let $K$ be a finite Galois extension of $\Q$. Let $H \subset \Gal(K/\Q)$ be a
conjugacy class. Then $$\Prob(\Frob(p)=H )=\frac{\#H}{\#\Gal(K/\Q)}.$$
\end{theo}

Before applying Chebotarev's theorem to the case of elliptic curves, we
introduce some notation. For every elliptic curve $E$ over a
field $F$ and all $m\in\mathbb{N}$, $m \ge 2$,  we consider the field $F(E[m])$ which is
the smallest extension of $F$ containing all the $m$-torsion of $E$.
The next result is classical, but we present its proof for the intuition it
brings.    

\begin{prop}\label{inj} For every integer $m\geq 2$ and any elliptic curve
$E$ over some field $F$, the following hold:
\begin{enumerate}
\item $F(E[m])/F$ is a Galois extension; 
\item there is an injective morphism $\iota_m:\Gal(F(E[m])/F)\hookrightarrow
\Aut(E(\overline{F})[m])$.
\end{enumerate}
\end{prop}
\begin{proof}
%\begin{enumerate}
%\item 
($1$) Since the addition law of $E$ can be expressed by rational functions
over $F$, there exist polynomials $f_m,g_m\in F[X,Y]$
such that the coordinates of the points in $E(\overline{F})[m]$ are the
solutions of the system $(f_m=0,g_m=0)$. Therefore $F(E[m])$ is the splitting
field of $\Res_X(f_m,g_m)$ and $\Res_Y(f_m,g_m)$ and in particular is Galois.
%\item

\noindent ($2$) For each $\sigma\in\Gal(F(E[m])/F)$ we call $\iota_m(\sigma)$ the application
which sends $(x,y)\in E(\overline{F})[m]$ into $(\sigma(x),\sigma(y))$.
Thanks to the discussion above, $\iota_m(\sigma)$ sends points of
$E(\overline{F})[m]$ in $E(\overline{F})[m]$. Since the addition law can be expressed by rational
functions over $F$, for each $\sigma$, $\iota_m(\sigma)\in\Aut(E(\overline{F})[m])$.
One easily checks that $\iota_m$ is a group morphism and its kernel is the
identity.
%\end{enumerate}
\end{proof}

\begin{nott}
We fix generators for $E(\overline{\Q})[m]$, thereby
inducing an isomorphism
${\psi_m: \Aut(E(\overline{\Q})[m]) \rightarrow \GL_2(\Z/m\Z)}$.
Let $\iota_m$ be the injection given by Proposition~\ref{inj}.  We call
$\rho_m:\Gal(\Q(E[m])/\Q)\rightarrow \GL_2(\Z/m\Z)$ the injective morphism 
$\psi_m\circ\iota_m$.

Let $p$ be a prime such that $E$ has good reduction at $p$ and $p\nmid m$.
Let $\iota_m^{(p)}$ be the
injection of $\Gal(\F_p(E[m])/\F_p)$ into $\Aut(E(\overline{\F_p})[m])$
given by Proposition
\ref{inj}. By~\cite[Prop.~VII.3.1]{Si09} there is a
canonical isomorphism $r_m^{(\mathfrak{p})}$ from $\Aut(E(\overline{\Q})[m])$ to
$\Aut(E(\overline{\F_p})[m])$ for each prime ideal $\mathfrak{p}$ over $p$.

\end{nott}

\begin{rem} \label{d} Note that $\#\Gal(\Q(E[m])/\Q)$ is bounded by
$\#\GL_2(\Z/m\Z)$. For every prime $\pi$, we have
$\#\GL_2(\Z/\pi\Z)=(\pi-1)^2(\pi+1)\pi$, and for every integer $k\geq 1$,
$\#\GL_2(\Z/\pi^{k+1}\Z)=\pi^4 \#\GL_2(\Z/\pi^k\Z)$. 
\end{rem}

\begin{nott} 
For all $g\in \GL_2(\Z/m\Z)$ we put  $\fix(g)=\{v\in (\Z/m\Z)^2\mid
g(v)=v \}$.
Conjugation of $g$ gives an isomorphic group of fixed elements.
If we are interested only in the isomorphism class we
use the notation $\fix(C)$ where $C$ is a set of conjugated 
elements.
We use analogous notations for
$\Aut(E(\overline{\Q})[m])$ and $\Aut(E(\overline{\F_p})[m])$.
\end{nott}

\begin{theo}[] \label{main}
Let $E$ be an elliptic curve over $\Q$ and $m \ge 2$ be an integer. Put
$K=\Q(E[m])$. Let $T$ be a subgroup of 
$\Z/m\Z \times \Z/m\Z$. Then,
\begin{enumerate}
\begin{comment}
\item For each $\sigma\in\Frob(\Q(E[m]),p)$, $\fix(\rho_m(\sigma))\simeq
\fix(\rho_m^{(p)}(\phi_p))$.
\end{comment}
\item   $\displaystyle \Prob(E(\F_p)[m] \simeq T) = \frac{\#\{g \in
\rho_m(\Gal(K/\Q)) \mid
\fix(g) \simeq T\}}{\#\Gal(K/\Q)}.$
\item  Let $a,n\in\mathbb{N}$ such that $a\leq n$ and
$\gcd(a,n)=1$ and let $\zeta_n$ be a primitive $n$th root of
unity. Put $G_a=\{\sigma\in \Gal(K(\zeta_n)/\Q)\mid
\sigma(\zeta_n)=\zeta_n^a\}$. Then:
\[
\Prob(E(\F_p)[m] \simeq T\mid p\equiv a \bmod n) = 
  \frac{\#\{\sigma \in G_a \mid 
        \fix(\rho_m(\sigma_{\vert K})) \simeq T\}}
       {\#G_a}.
\]
\end{enumerate}
\end{theo}
\begin{proof}
\begin{comment}
\begin{enumerate}
\item Let $\sigma\in\Gal(K/\Q)$ and define $\overline{\sigma}$ as in 
Definition
\ref{Frob}. We shall prove for all primes $p$ that $\rho_m(\sigma )$ and
$\rho_m^{(p)}(\overline{\sigma})$ are conjugated in $\GL_2(\Z/m\Z)$, in
particular their fixed subgroups are isomorphic. Let $\theta$ be
a primitive element of $\Q(E[m])$ and $\mu$ its minimal polynomial over $\Q$. 
Fix a prime $p$ not dividing $m$. Take
$\overline{\theta}$ a root of $\mu$ in $\overline{\mathbb{F}_p}$. We call reduction map
the application $r_{p,\theta,\overline{\theta}}:\Q(E[m])^2\rightarrow
\overline{\mathbb{F}_p}^2$, $(A(\theta),B(\theta))\mapsto  (\overline{A}(\overline{\theta}),\overline{B}(\overline{\theta}))$.
According to Proposition VII$.3.1$ in \cite{Si09}, for $p\nmid m$ such that
$E$ has good reduction,
$r_{p,\theta,\overline{\theta}}$ is injective and
$r_{p,\theta,\overline{\theta}}(E(\overline{\Q})[m])=
E(\overline{\mathbb{F}_p})[m]$. An easy computation shows that the action of
$\sigma$ on $E(\overline{\Q})[m]$ translates through
$r_{p,\theta,\overline{\theta}}$ into the action of $\overline{\sigma}$ on
$E(\overline{\mathbb{F}_p})[m]$. Therefore $\rho_m(\sigma)$ and
$\rho_m^{(p)}(\sigma)$ differ only due to the non canonical choice of $\rho_m$ and
$\rho_m^{(p)}$, thus by a conjugation in $\GL_2(\Z/m\Z)$.  
\end{enumerate}
\end{comment}
%\begin{enumerate}
%\item 
($1$) Let $p \nmid m$ be a prime for which $E$ has good reduction and let
$\mathfrak{p}$ be a prime ideal of $K$ over $p$.
We abbreviate $H=\{\sigma \in \Gal(K/\Q) \mid \fix(\iota_m(\sigma))
 \simeq T\}$.
First note that $E(\F_p)[m]=\fix(\iota_m^{(p)}(\phi_p))$ where $\phi_p$ is
the Frobenius in $\Gal(\F_p(E[m])/\F_p)$.
Since the diagram
%\begin{small}
\[
\xymatrix{\Dec({\mathfrak{p}\text{ }})  \ar@{^{(}->}[r]  \ar[d]^{\alpha^{(\mathfrak{p})}} &
           \Gal(\Q(E[m])/\Q)\text{ }\ar@{^{(}->}[r]^{\iota_m}                         &
         \Aut(E(\overline{\Q})[m]) \ar[d]^{r_m^{(\mathfrak{p})}}          \\
         \Gal(k_{\mathfrak{p}}/\F_p) \ar[r]^(0.4){\sim}                 &
         \Gal(\F_p(E[m])/\F_p)\text{ }\ar@{^{(}->}[r]^{\iota_m^{(p)}}                 &
          \Aut(E(\overline{\F_p})[m])                                \\
}
\]
%\end{small}
is commutative and since $\Frob(p) \subset \Gal(K/\Q)$
is the conjugacy class generated by
$(\alpha^{(\mathfrak{p})})^{-1}(\phi_{\mathfrak{p}})$
we have $E(\F_p)[m] \simeq \fix(\iota_m(\Frob(p)))$.

Decompose $H$ into a disjoint union of conjugacy classes $C_1,\ldots, C_N$.
Then $\fix(\iota_m(\Frob(p))) \simeq T$ is equivalent to
$\Frob(p)$ being one of the $C_i$.
Thanks to Theorem~\ref{Chebotarev} we obtain: 
\[
\Prob(E(\F_p)[m]\simeq T)  =  \sum_{i=1}^N\Prob(\Frob(p)=C_i) =
\sum_{i=1}^N\frac{\#C_i}{\#\Gal(K/\Q)} =  \frac{\#H}{\#\Gal(K/\Q)}.
\]

%\item 
\noindent ($2$) Using similar arguments as in ($1$) we have to evaluate
\[
\frac{\Prob(\Frob(p) \in \{C_1, \ldots, C_N\}, p\equiv a \bmod n)}
     {\Prob(p\equiv a \bmod n)}.
\]
Let $p$ be a prime and $\mathfrak{p}$ a prime ideal as
in the first part of the proof, and let
$\mathfrak{P}$ be a prime ideal of $K(\zeta_n)$ lying over $\mathfrak{p}$.
Furthermore let $\tilde{C}_1, \ldots, \tilde{C}_{\tilde{N}}$ be the conjugacy
classes of $\Gal(K(\zeta_n)/\Q)$ that are in the pre-images of
$C_1, \ldots, C_N$ and whose elements $\sigma$ satisfy
$\sigma(\zeta_n)=\zeta_n^a$.
Since $\Gal(K(\zeta_n)/\Q)$ maps $\zeta_n$ to primitive $n$th
roots of unity we have for
$\sigma\in(\alpha^{(\mathfrak{P})})^{-1}(\phi_{\mathfrak{P}})$ that
$\sigma(\zeta_n)=\zeta_n^b$ holds for some $b$.
Together with $\sigma(x)\equiv x^p\bmod \mathfrak{P}$ we get
$\zeta_n^b\equiv \zeta_n^p\bmod \mathfrak{P}$.
If we exclude the finitely many primes dividing the norms of $\zeta_n^c-1$ for
$c=1, \ldots, n-1$ we obtain $b\equiv p\bmod n$.
Since $\Frob(K(\zeta_n),p)$, the Frobenius conjugacy class for $K(\zeta_n)$,
is the pre-image of $\Frob(p)$, we get with the argument above
$\ds
\Prob(\Frob(p) \in \{C_1, \ldots, C_N\}, p\equiv a \bmod n)=
\Prob(\Frob(K(\zeta_n),p) \in \{\tilde{C}_1, \ldots, \tilde{C}_{\tilde{N}}\}).
$
A similar consideration for the denominator $\Prob(p\equiv a \bmod n)$ completes
the proof.
%\end{enumerate}
\end{proof}

\begin{rem}\label{BrierClavier} Put $K=\Q(E[m])$. If $[K(\zeta_n):\Q(\zeta_n)]=[K:\Q]$, then one has
$\Prob(E(\F_p)[m] \simeq T \mid p\equiv a \bmod n) = 
                                          {\Prob(E(\F_p)[m] \simeq T)}$
for $a$ coprime to $n$. Indeed, according to Galois theory,
 $\Gal(K(\zeta_n)/\Q)/\Gal(K(\zeta_n)/K) \simeq \Gal(K/\Q)$ through
$\overline{\sigma}\mapsto \sigma_{|K}$. Since
$[K(\zeta_n):\Q(\zeta_n)]=[K:\Q]$, we have $[K(\zeta_n):K]=\varphi(n)$ and
therefore each element $\sigma$ of
$\Gal(K/\Q)$ extends in exactly one way to an element of
$\Gal(K(\zeta_n)/\Q)$ which satisfies $\sigma(\zeta_n)=\zeta_n^a$. Note that
for $n\in\{3,4\}$ the condition is equivalent to $\zeta_n\not\in K$. 

The families constructed by Brier and Clavier~\cite{BrierC10}, which are
dedicated to integers $N$ such that the $n$th cyclotomic polynomial has
roots modulo all prime factors of $N$, modify $[K(\zeta_n)~:~\Q(\zeta_n)]$ by
imposing a large torsion subgroup over $\Q(\zeta_n)$.
\end{rem}

An important particular case of Theorem \ref{main} is as follows:
\begin{cor}\label{pi}
Let $E$ be an elliptic curve and $\pi$ be a prime number. Then,
\begin{enumerate}[]
\item $\ds\Prob(E(\F_p)[\pi] \simeq \Z/\pi\Z) = 
\frac {\#\{g \in \rho_\pi(\Gal(\Q(E[\pi])/\Q)) 
      \mid 
           %1 \in \Spec(g)
           %\text{$1$ is an eigenvalue of $g$}
           %1 \in \text{\rm Eig}(g)
           \det(g-\Id)=0, g\neq \Id
      \}}
      {\#\Gal(\Q(E[\pi])/\Q)}, $
\item $\ds\Prob(E(\F_p)[\pi] \simeq \Z/\pi\Z \times \Z/\pi\Z) = 
        \frac{1}{\#\Gal(\Q(E[\pi])/\Q)}.$
\end{enumerate}
\end{cor}
\begin{comment}
\begin{proof}
\begin{enumerate}
\item $\fix(g)=\ker(g)$ is a line if and only if
$1\in\Spec(g)$ and $g\neq \text{id}_{\Q(E(\pi))}$.
\item The only element in $\Gal(\Q(E[m])/\Q)$ which fixes $\mathbb{F}_\pi^2$ is the
identity.
\end{enumerate}
\end{proof}
\end{comment}
\begin{ex}
Let us compute these probabilities for the curves $E_1 : y^2=x^3+5x+7$ and 
$E_2 : y^2=x^3-11x+14$ and the primes $\pi=3$ and $\pi=5$. Here $E_1$ 
illustrates the generic case, whereas $E_2$ has special Galois groups. One 
checks with Sage~\cite{sage} that $[\Q(E_1[3]):\Q]=48$ 
and $\#\GL_2(\Z/3\Z)=48$. By Proposition~\ref{inj} we deduce that
$\rho_3(\Gal(\Q(E_1[3])/\Q))=\GL_2(\Z/3\Z)$. 
A simple computation shows that $\GL_2(\Z/3\Z)$ contains
$21$ elements having $1$ as eigenvalue, one of which is
$\Id$. Corollary \ref{pi} gives the following
probabilities: ${\Prob(E_1(\mathbb{F}_p)[3]\simeq \Z/3\Z)}=\frac{20}{48}$ and 
$\Prob(E_1(\mathbb{F}_p)[3]\simeq
\Z/3\Z\times \Z/3\Z)=\frac{1}{48}$. We used the same method for all the 
probabilities of Table \ref{ex1}, where we compare them to experimental
values.

\begin{table}[t]
  \begin{center}
	  \begin{tabular}{lc|r|r}
     & & \multicolumn{1}{c|}{$E_1$} & \multicolumn{1}{c}{$E_2$} \\
		\hline
    $\#\GL_2(\Z/3\Z)$ & & \multicolumn{2}{c}{$48$} \\
    \hline
    $\#\Gal(\Q(E[3])/\Q)$ & & \multicolumn{1}{c|}{$48$} & 
                              \multicolumn{1}{c}{$16$}\\
    \hline
    \multirow{2}{*}{$\Prob(E(\F_p)[3] \simeq \Z/3\Z \times \Z/3\Z)$} & Th. 
      \rule[-0.15cm]{0cm}{0.5cm} 
      & $\frac{1}{48} \approx 0.02083$ & $\frac{1}{16} = 0.06250$  \\
    & Exp. & $0.02082$ & $0.06245$ \\
    \hline
    \multirow{2}{*}{$\Prob(E(\F_p)[3] \simeq \Z/3\Z)$} & Th. 
      \rule[-0.15cm]{0cm}{0.5cm} 
      & $\frac{20}{48} \approx 0.4167 $ & $\frac{4}{16}=0.2500$ \\
    & Exp & $0.4165$ &  $0.2501$ \\
    \hline
    \hline
    $\#\GL_2(\Z/5\Z)$ & & \multicolumn{2}{c}{$480$} \\
    \hline
    $\#\Gal(\Q(E[5])/\Q)$ & & \multicolumn{1}{c|}{$480$} & 
                              \multicolumn{1}{c}{$32$}\\
    \hline
    \multirow{2}{*}{$\Prob(E(\F_p)[5] \simeq \Z/5\Z \times \Z/5\Z)$} & Th. 
      \rule[-0.15cm]{0cm}{0.5cm} 
      & $\frac{1}{480} \approx 0.002083$ & $\frac{1}{32} = 0.03125$ \\
    & Exp. & $0.002091$ & $0.03123$ \\
    \hline
    \multirow{2}{*}{$\Prob(E(\F_p)[5] \simeq \Z/5\Z)$} & Th.
      \rule[-0.15cm]{0cm}{0.5cm} 
      & $\frac{114}{480} = 0.2375 $ & $\frac{10}{32} = 0.3125$\\
    & Exp. & $0.2373$ &  $0.3125$ \\
		\end{tabular}
	\end{center}

	\caption{Comparison of the theoretical values (Th) of Corollary
\ref{pi} to the experimental results for all primes below $2^{25}$ (Exp). }
  \label{ex1}

\end{table}

Note that the relative difference 
between theoretical and experimental values never exceeds $0.4\%$. 
It is interesting to observe that reducing the Galois group
does not necessarily increase the probabilities, as it is shown for $\pi=3$.
\end{ex}

\subsection{Effective Computations of $\Q(E[m])$ and
$\rho_m(\Gal(\Q(E[m])/\Q))$ for Prime Powers.} \label{computepoldiv}
The main tools are the division polynomials as defined below. 
\begin{deff}\label{defdivpol} Let $E:y^2=x^3+ax+b$ be an elliptic curve over $\Q$ and $m\geq
2$ an integer. The $m$-division polynomial $P_m$ is defined as the monic polynomial
whose roots are the $x$-coordinates of all the $m$-torsion affine points.
$P_m^{\new}$ is defined as the monic polynomial whose roots are the $x$-coordinates
of the affine points of order exactly $m$.
\end{deff}

\begin{prop}\label{divpol} For all $m\geq 2$ we have:
\begin{enumerate}
\item $P_m,P_m^{\new} \in \Q[X]$;
\item $\deg(P_m)=\frac{(m^2+2-3\eta)}{2}$, where $\eta$ is the remainder of $m$ 
modulo $2$. 
%\item $\deg(P_m)=\Pi_{q|| m,\text{ $q$ prime power}}\frac{\varphi(q)(2q-\varphi(q))}{2^\eta}$ where $\eta$ is 
%the remainder of $q$ modulo $2$ and $\varphi$ is Euler's indicator. 
\end{enumerate}
\end{prop}

\begin{proof}
For a proof we refer to \cite{BlSeSm99}.
\end{proof}

Note that one obtains different division polynomials for other shapes of
elliptic curves (Weierstrass, Montgomery, Edwards, etc.).
Nevertheless, the Galois group ${\Gal(\Q(E[m])/\Q)}$ is model independent 
and can be computed with the division polynomials of Definition 
\ref{defdivpol} as, in characteristic different from $2$ and $3$, every curve 
can be written in short Weierstrass form.

One can compute $\Q(E[\pi])$ for any prime $\pi \geq 3$ using the following method:
\begin{enumerate}[1.]
\item Make a first extension of $\Q$ through an irreducible factor of 
$P_\pi$ and obtain a number field $F_1$ where $P_\pi$ has a root $\alpha_1$. 
\item Let $f_2(y)=y^2-(\alpha_1^3+a\alpha_1+b) \in F_1[y]$ and
$F_2$ be the extension of $F_1$ through $f_2$. $F_2$ contains a $\pi$-torsion
point $M_1$. In $F_2$, $P_\pi$ has $\frac{\pi-1}{2}$
trivial roots representing the $x$ coordinates of the multiples of $M_1$. 
\item Call $F_3$ the extension of $F_2$ through an irreducible factor of 
$P_\pi \in F_2[x]$ other than those corresponding to the trivial roots.
\item Let $\alpha_2$ be the new root of $P_\pi$ in $F_3$. Let 
$f_4(y)=y^2-(\alpha_2^3+a\alpha_2+b) \in F_3[y]$ and $F_4$ be the extension 
of $F_3$ through $f_4$. $F_4$ contains all the $\pi$-torsion. 
\end{enumerate}

The case of prime powers $\pi^k$ with $k\geq2$ is handled recursively.
Having computed $\Q(E[\pi^{k-1}])$, we obtain $\Q(E[\pi^k])$  by repeating the 4
steps above with $P_{\pi^k}^{\new}$ instead of $P_\pi$ and by 
considering as trivial roots all the $x$-coordinates of the points 
$\{P+M_1 \mid P \in E[\pi^{k-1}]\}$.

In practice, we observe that in general $P_\pi,f_2,P_\pi^{(F_2)}$ and $f_4$ are 
irreducible, where $P_\pi^{(F_2)}$ is $P_\pi$ divided by the factors
corresponding to the trivial roots. If this is the case, as 
$\deg(P_\pi)=\frac{\pi^2-1}{2}$ (Proposition \ref{divpol}),
the absolute degree of $F_4$ is $\frac{\pi^2-1}{2}\cdot 2\cdot
\frac{\pi^2-\pi}{2}\cdot 2=(\pi-1)^2(\pi+1)\pi$. By Remark \ref{d}, 
\text{$\#\GL_2(\Z/\pi\Z)=(\pi-1)^2(\pi+1)\pi$}, therefore in general we expect
$\rho_\pi(\Gal(\Q(E[\pi])/\Q))=\GL_2(\Z/\pi\Z)$. 
Also, we observed that in general the degree of the extension 
$\Q(E[\pi^k]) / \Q(E[\pi^{k-1}])$ is $\pi^4$.

Serre \cite{Se71} proved that the observations above are almost always true. 
The next theorem is a restatement of items $(1)$ and $(6)$ in the introduction 
of \cite{Se71}.
\begin{theo}[Serre] \label{Serre} Let $E$ be an elliptic curve without
complex multiplication.
\begin{enumerate}
\item For all primes $\pi$ and $k \geq 1$ the index
$[\GL_2(\Z/\pi^k\Z) : \rho_{\pi^k}(\Gal(\Q(E[\pi^k])/\Q))]$ 
is non-decreasing and bounded by a constant depending on $
E$ and $\pi$.
\item For all primes $\pi$ outside a finite set depending on $E$ and for all
$k\geq1$,\\ $\rho_{\pi^k}(\Gal(\Q(E[\pi^k])/\Q)=\GL_2(\Z/\pi^k\Z)$.
\end{enumerate}
 \end{theo}

\begin{deff} 
Put $I(E,\pi,k)=[\GL_2(\Z/\pi^k\Z):\rho_{\pi^k}(\Gal(\Q(E[\pi^k])/\Q))]$. If $E$
does not admit complex multiplication, we call Serre's exponent the integer
$n(E,\pi)=\min\{n\in\mathbb{N}^* \mid \forall k\geq n,
I(E,\pi,k+1)=I(E,\pi,k)\}$.
\end{deff}

In \cite{Se81} Serre showed that in some cases one can prove that
$I(E,\pi,k) = 1$ for all positive integers $k$. Indeed, Serre proved that the
surjectivity of $\rho_{\pi^k}$ (or equivalently $I(E,\pi,k)= 1$) follows from
the surjectivity of $\rho_{\pi}$ (or equivalently $I(E,\pi,1)=1$) for all
rational elliptic curves $E$ without complex multiplication and for all primes
$\pi \geq 5$.
In order to have the same kind of results for $\pi=2$
(resp. $\pi=3$) one has to suppose that $\rho_2$, $\rho_4$ and $\rho_8$ are 
surjective (resp. $\rho_3$ and $\rho_9$ are surjective).

Serre also conjectured that only a finite number of primes, not depending on the
curve $E$, can occur in the second point of Theorem \ref{Serre}. The current
conjecture is that for all rational elliptic curves without complex
multiplication and all primes $\pi \geq 37$, $\rho_\pi$ is surjective. In
\cite{Zyw11} Zywina described an algorithm that computes the primes $\pi$ for
which $\rho_\pi$ is not surjective and checked the conjecture for all elliptic
curves in Magma's database (currently this covers curves with conductor at most
$14000$). \footnote{Thanks to Andrew Sutherland for bringing this article
to our attention.}

The method described above allows us to compute $\Q(E[m])$ as an extension
tower. Then it is easy to obtain its absolute degree and a primitive element. 
Identifying $\rho_\pi(\Gal(\Q(E[m])/\Q))$ (up to conjugacy) is easy when
there is only one subgroup (up to conjugacy) of $\GL_2(\Z/m\Z)$
with the right order.
In the other case we check for each $g \in \GL_2(\Z/m\Z)$ using the fixed
generators of $E(\overline{\Q})[m]$ whether $g$ gives rise to an
automorphism on $\Q(E[m])$.
%Identifying $\rho_\pi(\Gal(\Q(E[m])/\Q))$ can be done in two ways:
%\begin{enumerate}
%\item when there is only one subgroup (up to conjugacy) of $\GL_2(\Z/m\Z)$ 
%with the right order, then there is nothing to compute;
%\item in the other case, we filter matrices of $\GL_2(\Z/m\Z)$ keeping only the
%ones that correspond to automorphisms of $\Q(E[m])$.
%\end{enumerate}
In practice, the bottleneck of this method is the factorization of
polynomials with coefficients  over number
fields.

An faster probabilistic algorithm for computing
$\Gal(\Q(E[\pi])/\Q)$ was proposed by Sutherland \cite{Su12}. This algorithm was
not known by the authors at the time of writing and would have helped to
accelerate the computation of the examples.

\subsection{Divisibility by a Prime Power.}
It is a common fact that, for a given prime $\pi$, the cardinality of an arbitrary
elliptic curve over $\F_p$ has a larger probability 
to be divisible by $\pi$ than an arbitrary integer of size $p$. In this
subsection we shall rigorously compute those probabilities under some
hypothesis of generality. 

\begin{nott} Let $\pi$ be a prime and $i,j,k\in\mathbb{N}$ such that $i\leq j$. 
We put:
\[p_{\pi,k}(i,j)=
\Prob(E(\F_p)[\pi^{k}]\simeq {\Z}/{\pi^i\Z} \times {\Z}/{\pi^j\Z}).
\]
Let $\ell \leq m$ be integers. When it is defined we denote:
\[p_{\pi,k}(\ell,m|i,j)=
\Prob(E(\F_p)[\pi^{k+1}] \simeq {\Z}/{\pi^\ell\Z} \times {\Z}/{\pi^m\Z} \mid 
E_p[\pi^{k}] \simeq {\Z}/{\pi^i\Z} \times {\Z}/{\pi^{j}\Z}).
\]
When it is clear from the context, $\pi$ is omitted.
\end{nott}

\begin{rem}\label{arbre}  
Since for every natural number $m$ and every prime $p$ coprime to $m$, 
$E(\mathbb{F}_p)[m] \subset {\Z/m\Z\times\Z/m\Z}$, we have $p_{\pi,k}(i,j)=0$ for $j>k$.
In the case $j<k$, if $p_{\pi,k}(\ell, m \mid i,j)$ is defined, it
equals $1$ if $(\ell,m)=(i,j)$ and equals $0$ if $(\ell,m) \neq (i,j)$.
Finally, for $j=k$, there are only three conditional 
probabilities which can be non-zero:
$p_{\pi,k}(i, k \mid i,k)$,
$p_{\pi,k}(i, k+1 \mid i,k)$, and
$p_{\pi,k}(k+1, k+1 \mid k,k)$.

\end{rem}

\begin{theo} \label{thprobcond}
Let $\pi$ be a prime and $E$ an elliptic curve over $\Q$.
If $k$ is an integer such that $I(E,\pi,k+1)=I(E,\pi,k)$, in
particular if $E$ has no complex multiplication and $k\geq n(E,\pi)$, then for
all $0\leq i<k$ we have:
\begin{enumerate}
%\item $p_\pi(0,1)=\frac{\pi^2-\pi-1}{\pi(\pi-1)^2}$;
\item $\ds p_{\pi,k}(k+1,k+1\mid k,k)=\frac{1}{\pi^4}$;
\item $\ds p_{\pi,k}(k,k+1\mid k,k)=\frac{(\pi-1)(\pi+1)^2}{\pi^4}$;
\item $\ds p_{\pi,k}(i,k+1\mid i,k)=\frac{1}{\pi}$.
\end{enumerate}
\end{theo}
\begin{proof}
Let $M=(\Z/\pi^k\Z)^2$. For all $g\in \GL_2(\pi M)$, 
we consider the set $\lift(g)=\{h\in \GL_2(M)\mid h_{\vert \pi M}=g\}
=\{g+\pi^{k-1} \left(\begin{smallmatrix} a & b \\ c & d\end{smallmatrix}\right)
\mid a,b,c,d\in\Z/\pi\Z\}$, whose cardinality is $\pi^4$. Since
$I(E,\pi,k+1)=I(E,\pi,k)$,
we have $\ds\frac{\#\Gal(\Q(E[\pi^k])/\Q)}{\#\Gal(\Q(E[\pi^{k+1}])/\Q)}=
\frac{\#\GL_2(\Z/\pi^k\Z)}{\#\GL_2(\Z/\pi^{k+1}\Z)}$, which equals
$\frac{1}{\pi^4}$ by Remark \ref{d}. So
for all $g\in \rho_{\pi^k}(\Gal(\Q(E[\pi^k])/\Q))$, $\lift(g)\subset
\rho_{\pi^{k+1}}(\Gal(\Q(E[\pi^{k+1}])/\Q))$.
Thanks to Theorem \ref{main}, the proof will follow if we count for each $g$
the number of lifts with a given fixed group.
\begin{enumerate}
\item For $g=\Id\in\rho_{\pi^k}(\Gal(\Q(E[\pi^k])/\Q)$, there is only one
element of $\lift(g)$ fixing
$(\Z/\pi^{k+1}\Z)^2$, so $p_{\pi,k}(k+1,k+1\mid k,k)= \frac{1}{\pi^4}$.

\item The element $g=\Id\in\rho_{\pi^k}(\Gal(\Q(E[\pi^k])/\Q)$, can be lifted
in exactly $\pi^4-1-\#\GL_2(\Z/\pi\Z)$ ways to elements in
$\GL_2(\Z/\pi^{k+1}\Z)$ which fix the $\pi^k$-torsion, a point of order
$\pi^{k+1}$, but not all the $\pi^{k+1}$-torsion. 
Therefore ${p_{\pi,k}(k,k+1\mid k,k)=\frac{(\pi-1)(\pi+1)^2}{\pi^4}}$.

\item Every element of $\GL_2(\Z/\pi^{k}\Z)$ which fixes a line, but is
not the identity, can be lifted in exactly $\pi^3$ ways to an
element of $\GL_2(\Z/\pi^{k+1}\Z)$ which fixes a line of
$(\Z/\pi^{k+1}\Z)^2$.
So ${p_{\pi,k}(i,k+1\mid i,k) = \frac{\pi^3}{\pi^4}=\frac{1}{\pi}}$.
\end{enumerate}
\end{proof}

The theorem below uses the information on $\Gal(\Q(E[\pi^{n(E,\pi)}])/\Q)$ for 
a given prime $\pi$ in order to
compute the probabilities of divisibility by any power of $\pi$. 

\begin{nott}
%Let $\pi$ be a prime and  $\delta(k)=\begin{cases}
%p_{i+1}(i+1,i+1) & \textrm{if } k=2i+1,\\ 0 & \text{otherwise.}\end{cases}$

%\noindent We also define
%\begin{align*}
%  S_k(h) & =\pi^k\left(\sum_{\ell=h}^{\lfloor \frac{k}{2}\rfloor}
%    p_{k-\ell}(\ell,k-\ell)+ \delta(k) \right), &
%  \gamma_n(h) & =\pi^n\sum_{\ell=0}^{h}{\pi^\ell p_n(\ell,n)}.
%\end{align*}
%\end{nott}

Let $\pi$ be a prime and 
$\displaystyle{\gamma_n(h) =\pi^n\sum_{\ell=0}^{h}{\pi^\ell p_n(\ell,n)}}$.
We also define
\[
\delta(k)= \begin{cases}
          p_{i+1}(i+1,i+1) & \textrm{if}\; k=2i+1\\ 
          0 & \text{otherwise}
\end{cases}, \, 
S_k(h) =\pi^k\left(\sum_{\ell=h}^{\lfloor \frac{k}{2}\rfloor}
    p_{k-\ell}(\ell,k-\ell)+ \delta(k) \right).
\]
\end{nott}

\begin{theo} \label{thprobetval}
Let $\pi$ be a prime, $E$ an elliptic curve over $\Q$, and $n$ be a positive
integer such that $\forall k \geq n$, $I(E,\pi,k)=I(E,\pi,n)$ (\textit{e.g.}, a
curve without complex multiplication and $n\geq n(E,\pi)$). 
Then, for any $k\geq1$,
\[\Prob(\pi^k\mid
\#E(\mathbb{F}_p))=\begin{cases}
\ds\frac{S_k(0)}{\pi^k} & \textrm{if } 1 \leq k\leq n, \\
\ds\frac{1}{\pi^k}(\gamma_n(k-n-1)+S_k(k-n)) & \textrm{if } n < k\leq 2n, \\
\ds\frac{1}{\pi^k}(\gamma_n(n)+p_n(n,n)\pi^{2n-1}-\frac{\pi^{4n-1}p_n(n,n)}{\pi^k})
& \textrm{if } k>2n.
\end{cases} 
\]
Let $\overline{v_\pi}$ be the average valuation of $\pi$ of $\#E(\F_p)$ for 
an arbitrary prime $p$. Then,
\[
\overline{v_\pi}=2\sum_{\ell=1}^{n-1}{p_\ell(\ell,\ell)}
                 +\frac{\pi}{\pi-1}\sum_{\ell=0}^{n-1}{p_n(\ell,n)}
                 +\sum_{\ell=0}^{n-2}{\sum_{i=\ell+1}^{n-1}{p_i(\ell,i)}}
                 +\frac{\pi(2\pi+1)}{(\pi-1)(\pi+1)}p_n(n,n).
\]
\end{theo} 
\begin{proof}
Let $k$ be a positive integer. Using Figure \ref{schema}, one checks that
\begin{align}\label{exprprob}
\Prob(\pi^{k} \mid \#E(\F_p)) & = \sum_{\ell=0}^{\lfloor \frac{k}{2} \rfloor}
  {p_{k-\ell}(\ell,k-\ell)} + \delta(k).
\end{align}

Let $c_1 = \frac{1}{\pi^4}$, 
  $c_2 =  \frac{(\pi-1)(\pi+1)^2}{\pi^4}$, 
  and $c_3 = \frac{1}{\pi}$. 
With these notations, the situation can be illustrated by Figure \ref{schema}. 
\begin{figure}[ht]
	\centering
	\includegraphics[scale=0.36]{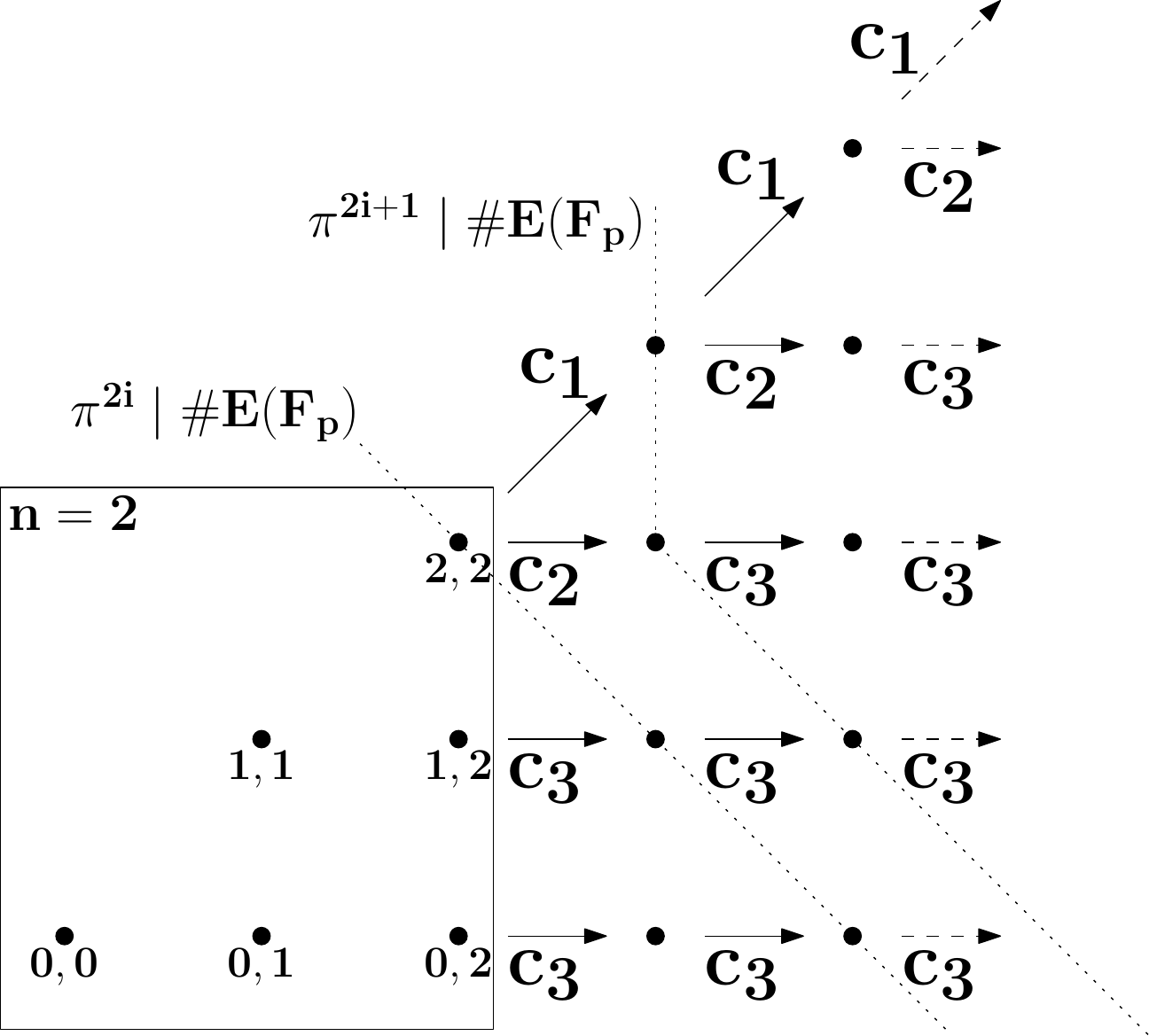}
  \caption{Each node of coordinates $(i,j)$ represents the event $\left ( 
  E_p[\pi^{j}]\simeq \Z/{\pi^i\Z} \times {\Z}/{\pi^j\Z}\right ).$
  The arrows represent the conditional probabilities of Theorem 
  \ref{thprobcond}.}
  \label{schema}
\end{figure}
For $j > n$ and $\ell < n$, the probability $p_j(\ell,j)$ is the product of the 
conditional probabilities of the unique path from $(\ell,j)$ to $(\ell,n)$ in 
the graph of Figure \ref{schema} times the probability $p_n(\ell,n)$. 
For $j > n$ and $\ell \geq n$, the probability $p_j(\ell,j)$ is the product of 
the conditional probabilities of the unique path from $(\ell,j)$ to $(n,n)$ 
in the graph of Figure \ref{schema} times the probability $p_n(n,n)$. 

There are three cases that are to be treated separately: $1 \leq k \leq n$, 
$n < k \leq 2n$ and $k > 2n$. For $1 \leq k \leq n$, the result follows from
Equation (\ref{exprprob}).
Let us explain the case for $k > 2n$, with $k=2i$:
\begin{align*}
 \Prob(\pi^{2i} \mid \#E(\F_p)) & = \sum_{\ell=0}^{i} 
                    {p_{2i-\ell}(\ell,2i-\ell)} + \delta(2i) 
                    = \sum_{\ell=0}^{i} {p_{2i-\ell}(\ell,2i-\ell)} \\
   & =\sum_{\ell=0}^{n-1}{p_{2i-\ell}(\ell,2i-\ell)} + 
      \sum_{\ell=n}^{i-1}{p_{2i-\ell}(\ell,2i-\ell)} + 
         p_i(i,i) \\ 
   & =\sum_{\ell=0}^{n-1}{c_3^{2i-l-n} p_{n}(\ell,n)} + 
      \sum_{\ell=n}^{i-1}{c_3^{2i-2l-1}c_2c_1^{l-n}p_n(n,n)} + 
         c_1^{i-n}p_n(n,n). \\ 
\end{align*}
After computations, one obtains the desired formula. 
The cases $k > 2n$ odd, and $n < k \leq 2n$ are treated similarly.
The formula for $\overline{v_\pi}$ is obtained using  
$\overline{v_\pi} = \sum_{k \ge 1}{\Prob(\pi^k \mid \#E_p)}.$

\end{proof}
\begin{rem} The theorem proves in particular that there exists a bound $B$
such that for primes $\pi>B$, $\Prob(\pi^2\mid \#E(\F_p))< \frac{2}{\pi^2}$,
so the probability that the cardinality is divisible by the square of a
prime greater than $B$ is at most
$\frac{2}{B}$. This confirms the experimental result that an elliptic curve
is close to a cyclic group when reduced modulo an arbitrary prime, regardless on
its rank over $\Q$.  
 \end{rem}
\begin{ex} \label{ex2} Let us compare the theoretical and experimental average 
valuation of $\pi=2$, $\pi=3$ and $\pi=5$ for the curve 
$E_1 : y^2=x^3+5x+7$ and $E_3 : y^2 = x^3 - 10875x + 526250$.
We exclude $E_2$ in this example since it has complex multiplication. 
For $E_1$, we apply Theorem \ref{thprobetval} with $n=1$ and compute the 
necessary probabilities with Corollary \ref{pi} knowing that the Galois groups 
are isomorphic to $\GL_2(\Z/\pi\Z)$. For $E_3$, we apply Theorem 
\ref{thprobetval} with $n=3$ for $\pi=2$ and $n=1$ for $\pi=3$ and $\pi=5$, and
compute the necessary probabilities with Corollary \ref{pi} (when $n=1$) and
Theorem \ref{main} (when $n=3$).The results are shown in Table
\ref{exthprobval}.
%E1:
%Theoretical average valuation of 3: 87/128=0.6796875000
%Experimental average valuation of 3: 0.6788994649
%Theoretical average valuation of 5: 695/2304=0.3016493056
%Experimental average valuation of 5: 0.3011130079
%E2.
%Theoretical average valuation of 3: 199/384=0.5182291667
%Experimental average valuation of 3: 0.5155953807
%Theoretical average valuation of 5: 355/768=0.4622395833
%Experimental average valuation of 5: 0.4685960281

\begin{table}[t]
	\begin{center}
		\begin{tabular}{c|c|c|c|c|c|c|c|c|c}
     & \multicolumn{3}{c|}{{\small Average valuation of $2$}} & 
       \multicolumn{3}{c|}{{\small Average valuation of $3$}} & 
       \multicolumn{3}{c}{{\small Average valuation of $5$}} \\ 
    & $n$ & Th. & Exp. & $n$ & Th. & Exp. & $n$ & Th. & Exp. \\
		\hline
    $ E_1$ \rule[-0.15cm]{0cm}{0.5cm} & 
                  $1$ & $\frac{14}{9} \approx 1.556$ & $1.555$ &
                  $1$ & $\frac{87}{128} \approx 0.680$ & $0.679$ &
                  $1$ & $\frac{695}{2304} \approx 0.302$ & $0.301$ \\
		\hline
    $E_3$ \rule[-0.15cm]{0cm}{0.5cm} & 
                  $3$ & $\frac{895}{576} \approx 1.554$ & $1.554$ &
                  $1$ & $\frac{39}{32} \approx 1.219$ & $1.218$ &
                  $1$ & $\frac{155}{192} \approx 0.807$ & $0.807$ \\
%		\hline
%    $E_2$ \rule[-0.15cm]{0cm}{0.5cm} & 
%                  $5$ & $\frac{1351}{384} \approx 3.518$ & $3.499$ &
%                  $2$ & $\frac{199}{384} \approx 0.518$ & $0.516$ &
%                  $1$ & $\frac{355}{768} \approx 0.462$ & $0.469$ \\
		\end{tabular}
	\end{center}

	\caption{Experimental values (Exp.) are obtained with all primes below 
  $2^{25}$. Theoretical values (Th.) come from Theorem~\ref{thprobetval}.}
  \label{exthprobval}
\end{table}

In order to apply Theorem~\ref{thprobetval}, one has to show that
$I(E,\pi,k)=I(E,\pi,n)$ for all $k \geq n$ (or $n\geq n(E,\pi)$ since $E_1$ and
$E_3$ have complex multiplication). For $E_1$, we were only able to prove that
$n(E,\pi)=1$ for $\pi=2$, $\pi=3$ and $\pi=5$ by using the remarks at the end 
of section \ref{computepoldiv}.
For $E_3$, Andrew Sutherland computed for us the Galois groups up to the
$2^5$-, $3^3$- and $5^2$-torsion. It is sufficient to compute the probabilities
and have some intuition for the values of $n$, but we were not able to prove
that they are correct.
In this case, we have to assume that the values of $n$ for which we were
able to compute the Galois group (and so the probabilities) are correct.

\end{ex}

\section{Applications to some Families of Elliptic Curves}\label{Applications}
As shown in the preceding section, changing the torsion properties is equivalent
to modifying the Galois group. One can see the fact of imposing rational
torsion points as a way of modifying the Galois group. In this section we
change the Galois group either by splitting the division polynomials or by
imposing some equations that directly modify the Galois group. With these ideas,
we find new infinite ECM-friendly families and we explain the properties of some
known curves. 

\subsection{Preliminaries on Montgomery and Twisted Edwards Curves.}\label{prelim}
Let $K$ be a field whose characteristic is neither $2$ nor $3$. 
\subsubsection{Edwards curves.}
For $a,d\in K$, with $ad(a-d)\ne 0$, the twisted Edwards curve $ax^2+y^2=1+dx^2y^2$
is denoted by $E_{a,d}$. The ``$a=-1$'' twisted Edwards curves are denoted 
by $E_d$. 
In~\cite{cryptoeprint:2008:016} completed twisted Edwards curves are 
defined by
$$\overline{E}_{a,d} = \{ ((X:Z), (Y:T))\in \PP^1\times \PP^1 \mid aX^2T^2 + Y^2Z^2 = Z^2T^2 + dX^2Y^2\}.$$
The completed points are the affine $(x,y)$ embedded into $\PP^1\times
\PP^1$
by $(x,y)\mapsto ((x~:~1),(y~:~1))$ (see~\cite{cryptoeprint:2008:016} for more
information).
We denote $(1:0)$ by $\infty$.
%If $x$ or $y$ is $(1:0)$ we denote this by $\infty$.

We give an overview of all the $2$- and $4$-torsion and some $8$-torsion 
points on $\overline{E}_{a,d}$, as specified in~\cite{cryptoeprint:2008:016},
in Figure~\ref{fig:edwardstorsion}.

\subsubsection{Montgomery curves and Suyama family.}
Let $A,B\in K$ be such that $B(A^2-4)\neq0$. 
The Montgomery curve $By^2=x^3+Ax^2+x$ associated to $(A,B)$ is denoted 
by $M_{A,B}$ (see~\cite{Montgomery:1987}) and its completion in $\PP^2$ by
$\overline{M}_{A,B}$.

\begin{rem}\label{Mont=TwEd}
If $a,d,A,B\in K$ are such that $d=\frac{A-2}{B}$ and $a=\frac{A+2}{B}$,
then there is a birational map between $\overline{E}_{a,d}$ and $\overline{M}_{A,B}$ 
given by $((x:z),(y:t))\mapsto ((t+y)x:(t+y)z:(t-y)x)$ 
(see~\cite{Bernstein2008Africacrypt}). Therefore $\overline{M}_{A,B}$ and
$\overline{E}_{a,d}$ have the same group structure over any field where defined and in particular 
the same torsion properties. Any statement in twisted Edwards language can be
easily translated into Montgomery coordinates and vice versa.
\end{rem}

A Montgomery curve for which there exist $x_3,y_3,k,x_\infty,y_\infty \in \Q$
such that
\begin{equation}\label{suyama}
   \left \{
   \begin{aligned}
   P_3(x_3) & = 0, &
   By_3^2 & = x_3^3+Ax_3^2+x_3 && \text{(3-torsion point)}\\
   k & = \frac{y_3}{y_\infty}, &
   k^2 & = \frac{x_3^3+Ax_3^2+x_3}{x_\infty^3+Ax_\infty^2+x_\infty} 
   && \text{(non-torsion point)}\\
   x_\infty & = x_3^3. & & && \text{(Suyama equation)} 
  \end{aligned}
   \right .
\end{equation}
is called a Suyama curve.
As described in~\cite{Suyama,DodZi06}, the solutions of \eqref{suyama} can be
parametrized by a rational value denoted $\sigma$. 
For all $\sigma \in \Q\backslash\{0,\pm 1, \pm 3, \pm 5, 
\pm \frac{5}{3}\}$, the associated Suyama curve has positive rank and a
rational point of order $3$. 

\begin{rem}
In the following, when we say that an elliptic curve $E_{a,d}$ has good 
reduction modulo a prime $p$, we also suppose that we have 
$v_p(a)=v_p(d)=v_p(a-d)=0$
(resp.~$v_p(A-2)=v_p(A+2)=v_p(B)=0$ for a Montgomery 
curve). In this case the reduction map is simply given by reducing the 
coefficients modulo $p$. The results below are also true for primes of good 
reduction which do not satisfy these conditions, by slightly modifying the 
statements and the proofs. Moreover, in ECM, if the conditions are not 
satisfied, we immediately find the factor $p$.
\end{rem}

\begin{figure}[t]
\centering
\begin{tikzpicture}[>=stealth',
    fact/.style={rectangle, draw=white, rounded corners=2mm, fill=white,
        text centered, anchor=center, text=black},
    level distance=0.8cm, growth parent anchor=south,
    kant2/.style={text width=2cm, text centered},
    empty/.style={circle,minimum size=0mm,text=black}
]
\tikzset{BarreStyle/.style =   {opacity=.25,line width=7 mm,line cap=round,color=#1}}

\node (R0) [fact] {$(0,1)$} [sibling distance=4cm] [<-]
child{ [sibling distance=2.4cm]
  node (R1) [fact] {$(0,-1)$}
  child{ 
    node (R2) [fact] {$(\pm \sqrt{a^{-1}},0)$}
    child{ 
      node (S0) [fact] {$(x_8,\pm \sqrt{a}x_8)$}
    }
  }
  child {
    node (R3) [fact] {$(\pm\sqrt{d^{-1}},\infty)$}
    child{ 
      node (S1) [fact] {$(\hat{x}_8,\pm \sqrt{d^{-1}} \hat{x}_8^{-1})$}
    }
  }  
}
child{
  node (R4) [fact] {$(\infty,\sqrt{\frac{a}{d}})$} [<-]
  child{
    node (R5) [fact] {$(\pm \sqrt{-a^{-1}}\sqrt[4]{\frac{a}{d}}, \pm\sqrt[4]{\frac{a}{d}})$}
  }
}
child{
  node (R6) [fact] {$(\infty,-\sqrt{\frac{a}{d}})$} [<-]
  child{
    node (R7) [fact] {$(\pm \sqrt{a^{-1}}\sqrt[4]{\frac{a}{d}}, \pm i\sqrt[4]{\frac{a}{d}})$}
  }
};
\draw [BarreStyle=gray] (R0.west) to (R0.east) ;
\draw [empty] (R0.north) node[empty,yshift=0.2cm,xshift=.8cm] {1-torsion};

\draw [BarreStyle=gray] (R1.west) to (R6.east) ;
\draw [empty] (R6.north) node[empty,yshift=0.2cm,xshift=.8cm] {2-torsion};

\draw [BarreStyle=gray] (R2.west) to (R7.east) ;
\draw [empty] (R7.north) node[empty,yshift=0.2cm,xshift=.8cm] {4-torsion};

\draw [BarreStyle=gray] (S0.west) to (S1.east) ;
\draw [empty] (S1.north) node[empty,yshift=0.2cm,xshift=.8cm] {8-torsion};

\end{tikzpicture}
\caption{An overview of all 1-, 2- and 4-torsion and some 8-torsion points 
on twisted Edwards curves.
The $x_8$ and $\hat{x}_8$ in the 8-torsion points are such that 
$adx_8^4-2ax_8^2+1=0$ and $ad\hat{x}_8^4-2d\hat{x}_8^2+1=0$.}
\label{fig:edwardstorsion}
\end{figure}
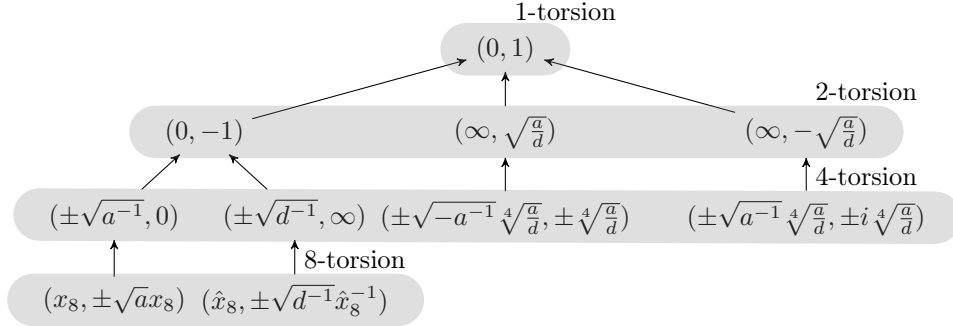

\subsection{Generic Galois Group of a Family of Curves.}
In the following, when we talk about the {\em Galois group of the $m$-torsion of
a family of curves}, we talk about a group isomorphic to the Galois group of the
$m$-torsion for all curves of the family except for a sparse set of curves 
(which can have a smaller Galois group).

For example, let us consider the Galois group of the $2$-torsion for the following
family $\{\mathcal{E}_r: y^2=x^3+rx^2+x \mid r \in \Q\backslash\{\pm2\} \}$. The Galois
group of the $2$-torsion of the
curve $\mathcal{E}: y^2=x^3+Ax^2+x$ over $\Q(A)$ is $\Z/2\Z$. Hence, for most
values of $r$ the Galois group is $\Z/2\Z$ and for a sparse set of values the
Galois group is the trivial group. So, we say that the Galois group of the
$2$-torsion of this family is $\Z/2\Z$.

To our best knowledge, there is no implementation of an algorithm computing 
Galois groups of polynomials with coefficients in a function field.
Instead we can compute the Galois group for every curve of the family, so we can
guess the Galois group of the family from a finite number of instantiations.
In practice, we took a dozen random curves in the family and if all these Galois groups of the $m$-torsion were the same, we guessed that it was the Galois 
group of the $m$-torsion of the family of curves.

\subsection{Study of the $2^k$-Torsion of Montgomery/Twisted Edwards Curves.}
\label{2ktorsion} The rational torsion of a Montgomery/twisted Edwards curve is
$\Z/2\Z$ but it is known that $4$ divides the order of the curve when reduced
modulo any prime $p$ \cite{Suyama}. The following theorem gives more detail on
the $2^k$-torsion. 

\begin{theo}\label{th2ktorsion}
Let $E=E_{a,d}$ be a twisted Edwards curve
(resp.~a Montgomery curve $M_{A,B}$)
over $\Q$. Let $p$ be a prime such that $E$ has good reduction at $p$.
\begin{enumerate}
\item If $p \equiv 3 \pmod{4}$ and $\frac{a}{d}$ (resp.~$A^2-4$) is a 
quadratic residue modulo $p$, then $E(\F_p)[4] \simeq \Z/2\Z \times \Z/4\Z$;
\item If $p \equiv 1 \pmod{4}$, $a$ (resp.~$\frac{A+2}{B}$) is a quadratic 
residue modulo $p$ (in particular $a=\pm1$) and $\frac{a}{d}$ 
(resp.~$A^2-4$) is a quadratic residue modulo $p$, then 
$\Z/2\Z \times \Z/4\Z \subset E(\F_p)[4]$;
\item If $p \equiv 1 \pmod{4}$, $\frac{a}{d}$ (resp.~$A^2-4$) is a quadratic 
non-residue modulo $p$ and $a-d$ (resp.~$B$) is a quadratic residue modulo $p$,
then $E(\F_p)[8] \simeq \Z/8\Z$.
\end{enumerate}
\end{theo}
\begin{proof}
Using Remark \ref{Mont=TwEd}, it is enough to prove the results in the 
Edwards language, which follow by some calculations using Figure 
\ref{fig:edwardstorsion}.
\end{proof}

Theorem~\ref{th2ktorsion} suggests that by imposing equations on the parameters
$a$ and $d$ we can improve the torsion properties. The case where $\frac{a}{d}$
is a square has been studied in \cite{cryptoeprint:2008:016} and \cite{starfish}
for the family of Edwards curves having $\Z/2\Z\times\Z/8\Z$ (when $a=1$)
respectively $\Z/2\Z\times\Z/4\Z$ (when $a=-1$) rational torsion.  Here we focus
on two other equations: 
\begin{align} 
\label{eqs11} & \exists c \in \Q,\, a=-c^2 & 
       \qquad \text{ (}A+2=-Bc^2 \text{ for Montgomery curves),} \\
\label{eqs94} & \exists c \in \Q,\, a-d=c^2 &
      \qquad  \text{ (}B=c^2 \text{ for Montgomery curves).}
\end{align}

The cardinality of the Galois group of the $4$-torsion for generic
Montgomery curves is $16$ and this is reduced to $8$ for the family of curves
satisfying \eqref{eqs11}.
%We compute the Galois group of the $4$-torsion of the family of curves
%satisfying \eqref{eqs11}, which have cardinality $8$ instead 
%of $16$ for the family of curves that do not satisfy \eqref{eqs11}. 
Using Theorem \ref{main}, we can compute the changes of
probabilities due to this new Galois group. For all curves satisfying
\eqref{eqs11} and all primes $p \equiv 1 \pmod{4}$, the probability of having
$\Z/2\Z \times \Z/2\Z$ as the $4$-torsion group becomes $0$ (instead of
$\frac{1}{4}$); the probabilities of having $\Z/2\Z \times \Z/4\Z$  and $\Z/4\Z
\times \Z/4\Z$ as the $4$-torsion group become $\frac{1}{4}$ (instead of
$\frac{1}{8}$). 

The Galois group of the $8$-torsion of the family of curves 
satisfying~\eqref{eqs94} has cardinality $128$ instead of $256$ for
generic Montgomery curves. Using
Theorem~\ref{main}, one can see that the probabilities of having an $8$-torsion
point are improved.

Using Theorem~\ref{thprobetval}, one can show that both families of curves,
the family satisfying~\eqref{eqs11} and the one satisfying~\eqref{eqs94},
%satisfying~\eqref{eqs11} and~\eqref{eqs94} 
increase the probability of having
the cardinality divisible by $8$ from $62.5\%$ to $75\%$ and the average 
valuation of $2$ from $\frac{10}{3}$ to $\frac{11}{3}$.

\subsection{Better Twisted Edwards Curves with Torsion $\ZTF$ using Division 
Polynomials. } \label{TwEd}
%The Suyama construction~\cite{Suyama} is traditionally used in the setting of
%Montgomery curves~\cite{Montgomery:1987} to construct a family of curves where the 
%group order $k$ is divisible by $12$. 
%This is the default choice in the GMP-ECM software package~\cite{GMPECM}.
In this section we search for curves such that some of the factors of the
division polynomials split and by doing so we try to change the Galois groups.
As an example we consider the family of $a=-1$ twisted Edwards curves $E_d$
with $\ZTF$-torsion, these curves are exactly the ones with $d=-e^4$ 
(see~\cite{starfish}). The technique might be used in any context. 

\subsubsection{Looking for subfamilies.}
%of Twisted Edwards Curves with Torsion Group Isomorphic to $\ZTF$} 

For a generic $d$, $P_8^{\new}$ splits into three irreducible 
factors: two of degree $4$ and one of degree $16$. If one takes $d=-e^4$, 
the polynomial of degree $16$ splits into three factors: two of 
degree $4$, called $P_{8,0}$ and $P_{8,1}$, and one of degree $8$, called 
$P_{8,2}$. By trying to force one of these three polynomials to split, 
we found four families, as shown in Table \ref{Edsubfamilies}.

\begin{table}[ht]
  \begin{center}
    \begin{tabular}{c|c|c|c|c|c}
    $d=-e^4$ & ``generic'' $e$ & $e=g^2$ & $e=\frac{2g^2+2g+1}{2g+1}$ & 
    $e=\frac{g^2}{2}$ & $e=\frac{g-\frac{1}{g}}{2}$ \\
    \hline

    degree of factors of $P_{8,0}$ & $4$ & $4$ & $4 $& $2,2$ & $2,2$ \\
    degree of factors of $P_{8,1}$ & $4$ & $4$ & $4$ & $4$ & $2,2$ \\
    degree of factors of $P_{8,2}$ & $8$ & $4,4$ & $4,4$ & $8$ & $8$  \\
    \hline
  average valuation of $2$ & $\frac{14}{3}$ & $\frac{29}{6}$ & 
                             $\frac{29}{6}$ & $\frac{29}{6}$ & 
                             $\frac{16}{3}$ \\
   for $p = 3 \bmod{4}$   & $4$ & $4$ & $4$ & $4$ & $5$\\ 
   for $p = 1 \bmod{4}$   & $\frac{16}{3}$ & $\frac{17}{3}$ & 
                             $\frac{17}{3}$ & $\frac{17}{3}$ & 
                             $\frac{17}{3}$ \\
    \end{tabular}
  \end{center}

  \caption{Subfamilies of twisted Edwards curves with torsion group isomorphic 
  to $\ZTF$ and the degrees of the irreducible factors of $P_{8,0}$, $P_{8,1}$
  and $P_{8,2}$.} 
  \label{Edsubfamilies}
\end{table} 
%\begin{table}[ht]
%  \begin{center}
%    \begin{tabular}{c|c|c|c}
%    \hline
%    $d=-e^4$ & $P_{8,0}$ & $P_{8,1}$ & $P_{8,2}$ \\
%    \hline
%    generic $e$ & $4$ & $4$ & $8$ \\
%    $e=g^2$ & $4$ & $4$ & $4,4$ \\
%    $e=\frac{2g^2+2g+1}{2g+1}$ & $4$ & $4$ & $4,4$ \\
%    $e=\frac{g^2}{2}$ & $2,2$ & $4$ & $8$ \\
%    $e=\frac{g-\frac{1}{g}}{2}$ & $2,2$ & $2,2$ & $8$
%    \end{tabular}
%  \end{center}
%
%  \caption{Subfamilies of twisted Edwards curves with torsion group isomorphic 
%  to $\ZTF$ and the degrees of the irreducible factors of $P_{8,0}$, $P_{8,1}$
%  and $P_{8,2}$.} 
%  \label{Edsubfamilies}
%\end{table} 

In all these families the generic average valuation of $2$ is increased by
$\frac{1}{6}$ ($\frac{29}{6}$ instead of $\frac{14}{3}$), except the family 
$e=\frac{g-\frac{1}{g}}{2}$ for which it is increased by $\frac{2}{3}$,
bringing it to the same valuation as for the family of twisted Edwards curves with
$a=1$ and torsion isomorphic to $\Z/2\Z \times \Z/8\Z$.
Note that these four families cover all the curves presented in the first 
three columns of \cite[Table 3.1]{starfish}, except the two curves with 
$e=\frac{26}{7}$ and $e=\frac{19}{8}$, which have a generic Galois group for the
$8$-torsion. 

\subsubsection{The family $e=\frac{g-\frac{1}{g}}{2}$. }
In this section, we study in more detail the family 
$e=\frac{g-\frac{1}{g}}{2}$. Using Theorem~\ref{main} one can prove that the 
group order modulo all primes is divisible by $16$. 
However, we give an alternative proof which is also of independent interest.
We need the following theorem which computes the
8-torsion points that double to the $4$-torsion points $(\pm \sqrt[4]{-d^{-1}},
\pm \sqrt[4]{-d^{-1}})$.

\begin{theo}
Let $E_d$ be a twisted Edwards curve over $\Q$ with
$d=-e^4$, $e=\frac{g-\frac{1}{g}}{2}$ and \mbox{$g\in\Q\setminus\{-1,0,1\}$}.
Let $p>3$ be a prime of good reduction.
If $t\in\{1,-1\}$ such that $tg(g-1)(g+1)$ is
a quadratic residue modulo $p$ then 
the points $(x,y) \in E_d(\F_p)$, with $w\in\{1,-1\}$,
and
\begin{equation}\label{eq:8torsion}
x=\pm g^wy, \qquad y=\pm \sqrt{\frac{4tg^{2-w}}{(g-tw)^3(g+tw)}}
\end{equation}
have order eight and double to $(\pm e^{-1}, t e^{-1})$.
\label{thm:edw1}\end{theo}
\begin{proof}
Note that all points $(x,y)$ of order eight satisfy \mbox{$\infty\ne x\ne 0\ne y\ne \infty$}. 
Following \cite[Theorem 2.10]{cryptoeprint:2008:016} a point $(x,y)$ doubles to 
$((2xy:1+dx^2y^2),(x^2+y^2:1-dx^2y^2))=((2xy:-x^2+y^2),(x^2+y^2:2-(-x^2+y^2)))$.
Let $s,t\in\{1,-1\}$ such that $(x,y)$ doubles to $(s e^{-1}, t e^{-1})$, hence
$$\frac{2xy}{-x^2+y^2}=\frac{s}{e} \quad\textrm{ and }\quad
\frac{x^2+y^2}{2-(-x^2+y^2)}=\frac{t}{e}.$$ 
From the terms in the first equation we obtain 
$\left(\frac{x}{y}\right)^2 + \frac{2exs}{y} + e^2 = 1+e^2$. 
Write $e=\frac{g-\frac{1}{g}}{2}$ such that 
$\left(\frac{x}{y} + se\right)^2 = \left(\frac{g+\frac{1}{g}}{2}\right)^2$.
Hence 
$\frac{x}{y}\in\left\{\pm g, \pm \frac{1}{g}\right\}$
depending on the sign $s$ and the sign after taking the square root.
This gives $x^2=G^2y^2$ with $G^2\in\{g^2,g^{-2}\}$.

From the second equation we obtain
$(e-t)x^2+(e+t)y^2 = 2t$ and
substituting $x^2$ results in 
$\left((e-t)G^2+(e+t)\right)y^2=2t$.
This can be solved for $y$ when $2t\left((e-t)G^2+(e+t)\right)$
is a quadratic residue modulo $p$.
This is equivalent to checking if any of
\begin{eqnarray}
2t\left((e-1)g^2+(e+1)\right) & = & \frac{t(g-1)^3(g+1)}{g}, \label{eq:d1}\\
2t\left((e-1)+(e+1)g^2\right) & = & \frac{t(g-1)(g+1)^3}{g}\label{eq:d2}
\end{eqnarray}
is a quadratic residue modulo $p$. 
By assumption $tg(g-1)(g+1)$ is a quadratic residue
modulo $p$. Hence, both expression~\eqref{eq:d1} and~\eqref{eq:d2} are quadratic residues
modulo $p$. Solving for $y$ and keeping track of all the signs
results in the formulae in \eqref{eq:8torsion}.
\begin{comment}
results in the following solutions (depending on $G^2$ and $t$)
$$\left\{\begin{array}{ll}
x=\pm \frac{y}{g}\quad & \textrm{if } G^2=g^{-2}\\
& \left\{\begin{array}{ll}
y=\pm \sqrt{\frac{-4g^3}{(g-1)^3(g+1)}}\quad & \textrm{if } t=-1\\
y=\pm \sqrt{\frac{4g^3}{(g-1)(g+1)^3}}  & \textrm{if } t=+1\\
\end{array}\right.\\
x=\pm yg & \textrm{if } G^2=g^2\\
& \left\{\begin{array}{ll}
y=\pm \sqrt{\frac{-4g}{(g-1)(g+1)^3}}\quad & \textrm{if } t=-1\\
y=\pm \sqrt{\frac{4g}{(g-1)^3(g+1)}}  & \textrm{if } t=+1\\
\end{array}\right.\\
\end{array}\right.$$
which finalizes the proof.
\end{comment}
\end{proof}

%\authnote[TODO]{We probably don't need that the numerator and denominator of 
%$g(g-1)(g+1)$ are co-prime to $p$. This should be fixed in a later version.}

A direct consequence of this theorem is as follows.
\begin{cor}\label{col:div}
Let $E=E_d$ be a twisted Edwards curve over $\Q$ with
$d=-\left(\frac{g-\frac{1}{g}}{2}\right)^4$,
\mbox{$g\in\Q\setminus\{-1,0,1\}$}
and
$p>3$ a prime of good reduction.
Then $E(\Q)$ has torsion group isomorphic to $\ZTF$ and
the group order of $E(\F_p)$ is divisible by $16$.  
\end{cor}
\begin{proof}
We consider two cases.
%\begin{enumerate}
%\item 

\noindent ($1$) If $p\equiv 1 \pmod 4$ then $-1$ is a quadratic residue 
modulo $p$. Hence, the 4-torsion points $(\pm i, 0)$ exist
(see Figure~\ref{fig:edwardstorsion}) and $16\mid\#E(\F_p)$.

%\item 
\noindent ($2$) If $p\equiv 3 \pmod 4$ then $-1$ is a quadratic non-residue modulo $p$. 
Then exactly one of $\{g(g-1)(g+1), -g(g-1)(g+1)\}$ is a quadratic residue 
modulo $p$. Using Thm.~\ref{thm:edw1} it follows that the curve $E(\F_p)$ 
has eight 8-torsion points and hence $16\mid\#E(\F_p)$.
%\end{enumerate}
\end{proof}
Corollary~\ref{col:div} explains the good behavior of the curve with $d=-(\frac{77}{36})^4$
and torsion group isomorphic to $\ZTF$ found in~\cite{starfish}. This parameter can
be expressed as $d=-(\frac{77}{36})^4=-\left(\frac{g-\frac{1}{g}}{2}\right)^4$ for 
$g=\frac{9}{2}$ and, therefore, the group order is divisible by an additional factor two.

\begin{cor}
Let $g\in\Q\setminus\{-1,0,1\}$, $d=-\left(\frac{g-\frac{1}{g}}{2}\right)^4$ and
$p\equiv 1 \pmod 4$ be a prime of good reduction.
If $g(g-1)(g+1)$ is a quadratic residue modulo $p$
then the group order of $E_d(\F_p)$ is divisible by $32$.
\end{cor}
\begin{proof}
All 16 4-torsion points are in $E_d(\F_p)$ (see Figure~\ref{fig:edwardstorsion}).
By Thm.~\ref{thm:edw1} we have at least one 8-torsion point. 
Hence, $32\mid\#E_d(\F_p)$.
\end{proof}
We generated different values $g\in\Q$ by setting
$g=\frac{i}{j}$ with $1\le i < j \le 200$ such that $\gcd(i,j)=1$.
This resulted in $12\,231$ possible values for $g$ and Sage~\cite{sage} found
$614$ non-torsion points. As expected, we observed that they behave similarly 
as the good curve found in~\cite{starfish}.

\subsubsection{Parametrization.}
In~\cite{starfish} a ``generating curve'' is specified which parametrizes $d$ and
the coordinates of the non-torsion points. Arithmetic on this curve can be used 
to generate an infinite family of twisted Edwards curves with torsion group 
isomorphic to $\ZTF$ and a non-torsion point.
Using ideas from~\cite{BrierC10}
we found a parametrization which does not involve a
generating curve and hence no curve arithmetic. 
\begin{theo}\label{parametrization}
Let $t\in\Q\setminus\{0,\pm 1\}$ and
$d=-e^4$, $e=\frac{3(t^2-1)}{8t}$, $x_{\infty}=(4e^3+3e)^{-1}$ and $y_{\infty}=\frac{9t^4-2t^2+9}{9t^4-9}$.
Then the twisted Edwards curve $-x^2+y^2=1+dx^2y^2$ has torsion group $\ZTF$ and
$(x_{\infty},y_{\infty})$ is a non-torsion point.
\end{theo}
\begin{proof}
The twisted Edwards curve has torsion group $\ZTF$ because $d=-e^4$ and $e$ is not 
equal to $0$ and $\pm 1$. The point $(x_{\infty},y_{\infty})$ is on the
curve and since $x_{\infty}\notin\{0,\infty,e^{-1},-e^{-1}\}$
this is a non-torsion point.
\end{proof}

This rational parametrization allowed us to impose additional conditions
on
the parameter $e$. For the four families, except $e=g^2$ which is treated below, the parameter
$e$ is given by an elliptic curve of rank $0$ over $\Q$. 

\begin{cor} Let $P=(x,y)$ be a non-torsion point on the elliptic curve
$y^2=x^3-36x$ having rank $1$. Let $t=\frac{x+6}{x-6}$,  using 
notations of Theorem~\ref{parametrization}, the curve $E_{-e^4}$ belongs
 to the family $e=g^2$ and has positive rank over $\Q$. 
\end{cor}

\subsection{Better Suyama Curves by a Direct Change of the Galois Group.}
\label{Suyama11}
In this section we will present two families that change the Galois group of the
$4$- and $8$-torsion without modifying the factorization pattern of the 
$4$- and $8$-division polynomial.

\subsubsection{Suyama-$11$.}\label{S11}
Kruppa observed in \cite{Kru10} that among the Suyama curves, the one
corresponding to $\sigma=11$ finds exceptionally many primes. Barbulescu
\cite{Bar09} extended it to an infinite family that we present in detail
here. 

%Table \ref{compare} presents an experiment on the special properties of 
%$\sigma=11$ and
%$\sigma=\frac{9}{4}$. We compare the results to $\sigma=12$ which seems to
%be a generic Suyama curve. 
%\begin{table}[ht]
%  \begin{center}
%    \begin{tabular}{|l|l|l|l|}
%    \hline
%    $\sigma$ & average valuation of $2$ & $p \equiv 1 \pmod{4}$ 
%                                        & $p \equiv 3 \pmod{4}$\\
%    \hline
%    $12$& $3.33$ & $3.16$ & $3.50$ \\
%    $11$ & $3.67$ & $3.83$ & $3.50$ \\
%    $\frac{9}{4}$ & $3.66$ & $3.83$ & $3.50$ \\ 
%    \hline
%    \end{tabular}
%  \end{center}
%
%  \caption{ Test sample: all primes under $2^{23}$. } 
%  \label{compare}
%\end{table} 

Experiments show that the $\sigma=11$ curve differs from other
Suyama curves only by its probabilities to have a given $2^k$-torsion when
reduced modulo primes $p \equiv 1 \pmod{4}$.
The reason is that the $\sigma=11$ curve satisfies Equation \eqref{eqs11}. 
Section \ref{2ktorsion} illustrates the changes in probabilities 
of the $\sigma=11$ curve when compared to curves which do not satisfy
Equation~\eqref{eqs11} and shows that Equation \eqref{eqs11} improves the 
average valuation of $2$ from $\frac{10}{3}$ to $\frac{11}{3}$.

Let us call Suyama-$11$ the set of Suyama curves which satisfy Equation
\eqref{eqs11}. When solving the system formed by Suyama's system plus
Equation~\eqref{eqs11}, we obtain an elliptic parametrization for $\sigma$. Given a point
$(u,v)$ on $E_{\sigma_{11}} : v^2 = u^3 - u^2 - 120u + 432$, $\sigma$ is
obtained as $\sigma = \frac{120}{u-24}+5$.
The group $E_{\sigma_{11}}(\Q)$ is generated by the points
$P_\infty=(-6,30)$, $P_2=(-12,0)$ and
$Q_2=(4,0)$ of orders $\infty$, $2$ and $2$ respectively. We exclude $0, \pm P_\infty$, $P_2$, $Q_2$, $P_2+Q_2$, and $Q_2\pm
P_\infty$, which are the points producing non-valid values of $\sigma$.
The points $\pm R, Q_2 \pm R$ lead to isomorphic curves.
Note that the $\sigma=11$ curve corresponds to the point
$(44,280)=P_\infty+P_2$.

\subsubsection{Edwards $\Z/6\Z$: Suyama-$11$ in disguise.}
In~\cite[Sec.~$5$]{starfish} it is shown that the $a=-1$ twisted Edwards curves
with $\Z/6\Z$-torsion over $\Q$ are precisely the curves $E_d$ with $d=\frac{-16u^3(u^2-u+1)}{(u-1)^6(u+1)^2}$ where
$u$ is a rational parameter.\footnote{There is a typo in the proof 
of~\cite[Thm.~5.1]{starfish}; the $\frac{16u^3(u^2-u+1)}{(u-1)^6(u+1)^2}$
misses a minus sign.} In particular, according to \cite[Sec.~$5.3$]{starfish}
one can translate any Suyama curve in Edwards language and then impose 
the condition that $-a$ is a square to obtain curves of the $a=-1$ type.
Finally, \cite[Sec.~$5.5$]{starfish} points out that this family has exceptional torsion properties. 

In order to understand the properties of this family, we translate it back
to Montgomery language using Remark~\ref{Mont=TwEd}. Thus, we are interested
in Suyama curves which satisfy equation $A+2=-Bc^2$ (the Montgomery equivalent 
for $-a$ being a square). This is the Suyama-$11$ family, so its torsion
properties were explained in Section~\ref{S11}.  
These two families have been discovered independently in~\cite{Bar09} 
and~\cite{starfish}. 

\subsubsection{Suyama-$\frac{9}{4}$.}\label{Suyama94}
In experiments by Zimmermann, new Suyama curves with exceptional
torsion properties were discovered, such as $\sigma=\frac{9}{4}$.
Further experiments
show that their special properties are related to the
$2^k$-torsion and concern exclusively primes $p \equiv 1 \pmod{4}$. 
Indeed, the $\sigma=\frac{9}{4}$ curve satisfies Equation \eqref{eqs94}. 
Section \ref{2ktorsion} illustrates the changes in probabilities 
of the $\sigma=\frac{9}{4}$ curve when compared to curves which do not satisfy
Equation~\eqref{eqs94} and shows that Equation \eqref{eqs94} improves the 
average valuation of $2$ from $\frac{10}{3}$ to $\frac{11}{3}$.

Let us call Suyama-$\frac{9}{4}$ the set of Suyama curves which satisfy Equation
\eqref{eqs94}. When solving the system formed by Suyama's system plus
Equation~\eqref{eqs94}, we obtain an elliptic parametrization for $\sigma$. Given a point
$(u,v)$ on $E_{\sigma_{94}} : v^2 = u^3 - 5u$, $\sigma$ is
obtained as $\sigma=u$.  The group $E_{\sigma_{94}}(\Q)$
is generated by the points $P_\infty=(-1,2)$ and $P_2=(0,0)$ of orders $\infty$ and $2$
respectively. We exclude the points $0,\pm
P_\infty$, $P_2$ and $P_2\pm P_\infty$ which produce non-valid values of
$\sigma$. If two points in $E_{\sigma_{94}}(\Q)$ differ by $P_2$ they correspond
to isomorphic curves. We recognize the curve associated to
$\sigma=\frac{9}{4}$ when considering the point $(\frac{9}{4},-\frac{3}{8}) =
[2]P_\infty$.

\begin{comment}
\subsection{Theoretical comparison of old and new families}

As shown by Table~\ref{compare}, the Suyama-$11$/Edwards $\Z/6\Z$ family has
the same behavior as the family Montgomery $\Z/12\Z$/ Edwards $\Z/12\Z$
presented in~\cite[Sec.~$10.3.2$]{Montgomery:1987} and 
in~\cite[Sec.~$8.1$]{cryptoeprint:2008:016}. Hence, the
practical efficiency is determined mainly by the arithmetics of each shape.

\begin{table}[ht]
\begin{center}
\begin{tabular}{c|ccccc}
Curve & \multicolumn{5}{c}{$\Prob(2^k|\#E(\F_p))$ in percent} \\
      & 2 & 3 & 4 & 5 & 6 \\
\hline
$\sigma=11$ & $100$ & $75$ & $43.8$ & $23.4$ & $12.1$\\
$\sigma=\frac{9}{4}$ & $100$ & $75$ & $43.8$ & $23.4$ & $12.1$\\
$\sigma=12$  &  $100$ & $62.5$ &  $34.4$ & $18.0$ & $9.2$ \\
$M\text{ }\Z/12\Z$ & $100$ & $75$ & $43.8$ & $23.4$ & $12.1$\\
\end{tabular}
\end{center}
\caption{We write $\sigma=\sigma_0$ for the Suyama curve with a given value
$\sigma_0$. $M$~$\Z/12\Z$ stands for $M_{A,B}$ with
$(A,B)=(\frac{2399}{3200},  \frac{123201}{40960000} )$
which represents the Montgomery curves of torsion $\Z/12\Z$. Probabilities
were rounded with a precision of $0.1$ percent.  }
\label{compare}
\end{table}

Comparing the Edwards $\Z/6\Z$ and Edwards $\Z/2\Z\times \Z/8\Z$ families
requires proving the independence of the $2^k$ and $3^\ell$-torsions. Point
$(3)$ in~\cite{Ser71} suggests that this can be done by computing $\Q(E[m])$
for a well chosen value of $m$. Nevertheless, experiments support this
hypothesis allowing us to establish Table~\cite{6vs16}.
\end{comment}

\subsection{Comparison.} Table \ref{comparison} gives a summary of all 
the families discussed in this article. The theoretical average valuations were
computed with Theorem \ref{thprobetval}, Theorem \ref{main} and Corollary 
\ref{pi} under some assumptions on Serre's exponent
(see Example \ref{ex2} for more information).
%\emph{Thorsten writes here.} 
\begin{table}[t]
	\begin{center}
		\begin{tabular}{c|c||c|c|c||c|c|c}
    \multirow{2}{*}{Families} & \multirow{2}{*}{Curves} & 
    \multicolumn{3}{c||}{{\small Average valuation of $2$}} & 
       \multicolumn{3}{c}{{\small Average valuation of $3$}} \\ 
     & & $n$ & Th. & Exp. & $n$ & Th. & Exp. \\
		\hline
    Suyama & $\sigma=12$ \rule[-0.15cm]{0cm}{0.5cm} & 
                  $2$ & $\frac{10}{3}\approx 3.333$ & $3.331$ & 
                  $1$ & $\frac{27}{16}\approx 1.688$ & $1.689$ \\
		\hline
    Suyama-$11$ & $\sigma=11$ \rule[-0.15cm]{0cm}{0.5cm} & 
                  $2$ & $\frac{11}{3}\approx 3.667$ & $3.669$ & 
                  $1$ & $\frac{27}{16}\approx 1.688$ & $1.687$ \\
		\hline
    Suyama-$\frac{9}{4}$ & $\sigma=\frac{9}{4}$ \rule[-0.15cm]{0cm}{0.5cm} & 
                  $3$ & $\frac{11}{3}\approx 3.667$ & $3.664$ &
                  $1$ & $\frac{27}{16}\approx 1.688$ & $1.687$ \\
		\hline
		\hline
    $\Z/2\Z \times \Z/4\Z$ & $E_{-{11}^4}$ \rule[-0.15cm]{0cm}{0.5cm} & 
                  $3$ & $\frac{14}{3}\approx 4.667$ & $4.666$ &
                  $1^*$ & $\frac{87}{128}\approx 0.680$ & $0.679$ \\
		\hline
    $e=\frac{g-\frac{1}{g}}{2}$ & 
            $E_{-{\left(\frac{77}{36}\right)}^4}$ \rule[-0.15cm]{0cm}{0.5cm} & 
                  $3$ & $\frac{16}{3}\approx 5.333$ & $5.332$ &
                  $1^*$ & $\frac{87}{128}\approx 0.680$ & $0.679$ \\
		\hline
    $e=g^2$ & 
            $E_{-{9}^4}$ \rule[-0.15cm]{0cm}{0.5cm} & 
                  $3$ & $\frac{29}{6}\approx 4.833$ & $4.833$ &
                  $1^*$ & $\frac{87}{128}\approx 0.680$ & $0.680$ \\
		\hline
    $e=\frac{g^2}{2}$ & 
            $E_{-{\left(\frac{81}{8}\right)}^4}$ \rule[-0.15cm]{0cm}{0.55cm} & 
                  $3$ & $\frac{29}{6}\approx 4.833$ & $4.831$ &
                  $1^*$ & $\frac{87}{128}\approx 0.680$ & $0.679$ \\
		\hline
    $e=\frac{2g^2+2g+1}{2g+1}$ & 
            $E_{-{\left(\frac{5}{3}\right)}^4}$ \rule[-0.15cm]{0cm}{0.55cm} & 
                  $3$ & $\frac{29}{6}\approx 4.833$ & $4.833$ &
                  $1^*$ & $\frac{87}{128}\approx 0.680$ & $0.679$ \\
		\end{tabular}
	\end{center}

	\caption{Experimental values (Exp.) are obtained with all primes below 
  $2^{25}$. The case $n=1^*$ means that the Galois group is isomorphic to 
  $GL_2(\Z/\pi\Z)$.}
  \label{comparison}
\end{table}

Note that, when we impose torsion points over $\Q$, the average valuation does
not simply increase by $1$, as can be seen in Table~\ref{comparison} for the
average valuation of $3$. 

\section{Conclusion and Further Work}
We have used Galois theory in order to analyze the torsion properties of
elliptic curves. We have determined the behavior of generic elliptic
curves and explained the exceptional properties of some known curves
(Edwards curves of torsion $\Z/2\Z\times \Z/4\Z$ and $\Z/6\Z$). The new
techniques suggested by the theoretical study have helped us to find infinite
families of curves having
exceptional torsion properties. We list some questions which were not
addressed in this work:
\begin{itemize}
%\item Given a curve $E$ over $\Q$ and a prime $\pi$, can one effectively compute Serre's exponent $n(E,\pi)$?
\item How does Serre's work relate to the independence of the $m$- and $m'$-torsion probabilities for
coprime integers $m$ and $m'$?
\item Is there a model predicting the success probability of ECM from
the probabilities given in Theorem \ref{thprobetval}?
\item Is it possible to effectively use the Resolvent Method \cite{Co93} in
order to compute equations which improve the torsion properties?
\end{itemize}

%\nocite{montgomery1992fft}
\bibliographystyle{abbrv}
\bibliography{article}
\end{document}